\documentclass{amsproc}
\usepackage{amsmath,amsthm, amscd, amssymb, amsfonts}
\usepackage[all]{xy}
\usepackage{hhline}
\usepackage[dvips, dvipsnames, usenames]{color}
\usepackage{hyperref}        
\usepackage{mathtools}
\usepackage[active]{srcltx}

\usepackage{ifpdf}
\ifpdf
\usepackage[pdftex]{graphicx}
\else
\usepackage[dvips]{graphicx}
\fi
\usepackage{tikz}
 	 \usetikzlibrary{arrows,backgrounds}
\usepackage[all]{xy}

\usepackage{yfonts}
\usepackage{fullpage}

\usepackage{bigstrut}

\usepackage{xcolor}
\definecolor{dark-red}{rgb}{0.7,0.25,0.25}
\definecolor{dark-blue}{rgb}{0.15,0.15,0.55}
\definecolor{medium-blue}{rgb}{0,0,.8}
\definecolor{DarkGreen}{RGB}{0,150,0}
\hypersetup{
   colorlinks, linkcolor={purple},
   citecolor={medium-blue}, urlcolor={medium-blue}
}

\newcommand{\nc}{\newcommand}
\newcommand{\ot}{\otimes}

\newcommand{\co}{\operatorname{co}}

\newcommand{\ord}{\operatorname{ord}}

\nc{\ydk}{^{K}_{K}\mathcal{YD}}

\newcommand\toba{{\mathfrak B }}

\newcommand{\gr}{\operatorname{gr}}

\nc{\Tr}{\mathrm{Tr}}
\nc{\X}{\mathbf{X}}

\newcommand{\Z}{{\mathbb Z}}
\newcommand{\N}{{\mathbb N}}

\newcommand{\Dih}{{\mathbb D}}

\newcommand{\matha}{\mathcal{A}}

\newcommand{\Ga}{{\Gamma}}
\newcommand{\M}{{\mathcal M}}

\newcommand{\Sc}{{\mathcal S}}

\nc{\eps}{\varepsilon}

\newcommand{\Hc}{{\mathcal H}}

\newcommand{\Oc}{{\mathcal O}}
\newcommand{\oc}{{\mathcal O}}

\newcommand{\yd}{\mathcal{YD}}
\newcommand{\ydho}{{}^{H_{0}}_{H_{0}}\mathcal{YD}}

\newcommand{\ydgap}{{}^{\Gamma_{4p}}_{\Gamma_{4p}}\mathcal{YD}}

\newcommand{\ydg}{{}^{G}_{G}\mathcal{YD}}

\newcommand{\com}{\Delta}

\newcommand{\Ext}{\operatorname{Ext}}

\newcommand\ad{\operatorname{ad}}
\newcommand\Hom{\operatorname{Hom}}

\newcommand{\Comod}{\mbox{\rm Comod\,}}

\nc{\coM}{\M^\ast(2,\Bbbk)}
\nc{\coMtres}{\M^\ast(3,\Bbbk)}
\nc{\coMcua}{\M^\ast(4,\Bbbk)}
\nc{\coMcin}{\M^\ast(5,\Bbbk)}
\nc{\coMt}{\M^\ast(t,\Bbbk)}
\nc{\coMj}{\M^\ast(j,\Bbbk)}
\nc{\coMn}{\M^\ast(n,\Bbbk)}
\nc{\coMd}{\M^\ast(d,\Bbbk)}
\nc{\GH}{G(H)}
\nc{\mas}{\oplus}
\nc{\cA}{\mathcal{A}}
\nc{\PH}{\cP(H)}
\nc{\e}{\varepsilon}
\nc{\GL}{\operatorname{GL}}
\nc{\wact}{\rightharpoonup}
\nc{\cark}{char\,k}
\nc{\adl}{\ad_\ell}
\nc{\cP}{\mathcal{P}}
\nc{\cU}{\mathcal{U}}
\nc{\fD}{\mathfrak{D}}
\nc{\cE}{\mathcal{E}}
\nc{\cS}{\mathcal{S}}
\nc{\be}{\textbf{e}}

\newcommand\id{\operatorname{id}}

\newcommand\Dic{\operatorname{Dic}}

\newcommand{\comment}[1]{}

\renewcommand{\a}{\alpha}
\renewcommand{\b}{\beta}
\def\o{\otimes}
\def\w{\omega}
\def\inv{^{-1}}
\def\l{\lambda}
\def\k{\Bbbk}
\def\DD{\mathcal{D}}
\def\ol{\overline}
\DeclareMathOperator{\ann}{ann}
\DeclareMathOperator{\Soc}{Soc}
\DeclareMathOperator{\ev}{ev}
\DeclareMathOperator{\db}{db}

\def\pf{\begin{proof}}
\def\epf{\end{proof}}

\theoremstyle{remark}

\newtheorem{thm}{Theorem}[section]
\newtheorem{lem}[thm]{Lemma}
\newtheorem{cor}[thm]{Corollary}

\newtheorem{prop}[thm]{Proposition}

\theoremstyle{definition}
\newtheorem{definition}[thm]{Definition}

\newtheorem{exas}[thm]{Examples}

\theoremstyle{remark}
\newtheorem{remark}[thm]{Remark}
\newtheorem{obs}[thm]{Remark}

\numberwithin{equation}{section}

\theoremstyle{plain}
\newcounter{maint}

\theoremstyle{plain}

\def\kG{\Bbbk G}
\def\k{\Bbbk}
\def\s{\sigma}

\newcommand{\ydK}{{}^K_K\mathcal{YD}}

\begin{document}

\renewcommand{\baselinestretch}{1.2}

\thispagestyle{empty}

\title[Dimension $8p$ with the Chevalley property]
{Nonsemisimple Hopf algebras of dimension $8p$ with the Chevalley property}
\author{Margaret Beattie}
\address{Mount Allison University,
Sackville, NB E4L 1E6, Canada}
\email{mbeattie@mta.ca} 
\author{Gast\'on A. Garc\'ia}
\address{Departamento de Matem\'atica, Facultad de Ciencias Exactas,
Universidad Nacional de La Plata. CONICET. Casilla de Correo 172, (1900)
La Plata, Argentina}
\email{ggarcia@famaf.unc.edu.ar}
\thanks{The second author was partially supported by
 ANPCyT-Foncyt, CONICET, 
 and Secyt (UNLP). The third author was partially supported by the NSF grant DMS 1501179. }
\author{Siu-Hung Ng}
\address{Department of Mathematics, Louisiana State University, Baton Rouge, LA 70803, USA} 
\email{rng@math.lsu.edu}
\author{Jolie Roat}
\address{\noindent J.R.: Mathematics Department, SUNY Cortland, Cortland, NY 13045, USA}
\email{jolie.roat@cortland.edu}

\date{\today}

\begin{abstract}
Let $p$ be an odd prime  and $\Bbbk$ an algebraically closed field of characteristic zero.
We classify nonsemisimple Hopf algebras over $\Bbbk$ of dimension $8p$ with the Chevalley property,   
and give partial results on nonsemisimple Hopf algebras of dimension $24$.
\end{abstract}

\maketitle

\section{Introduction}
It was conjectured by Kaplansky \cite{Kapl75} that there are only finitely many 
isomorphism classes of finite-dimensional Hopf algebras over $\k$ for any given dimension. 
This conjecture was shown to be false by counterexamples of dimension $p^4$ with $p$ an odd prime
\cite{AS98,BDG99,Gel98}, but it is true for dimension 16 \cite{GV}. However, the module categories over 
Hopf algebras of dimension 32 yield infinitely many inequivalent finite tensor categories \cite{EG02}. 
In particular, there are infinitely isomorphism classes of Hopf algebras of dimension 32, see \cite{B00}, \cite{Gr3}.

In the past two decades, Hopf algebras in a dimension which is factorized with \textit{few} primes
have been classified (cf. \cite{BG2}). As a consequence, the number of isomorphsim classes of Hopf 
algebras of dimensions less than 32 but not equal to 24 is known to be finite. 
Semisimple Hopf algebras of dimension $24$ 
with a $C^*$-structure  were classified in \cite{IK}, but the general classification without $C^*$-structure remains 
open. In the nonsemisimple case, pointed Hopf algebras of dimension 24 are classified and there are finitely 
many of them up to isomorphism. However, not all the 24-dimensional Hopf algebras are pointed. 

The family of Hopf algebras with the Chevalley property 
consists of Hopf algebras whose coradical forms a Hopf subalgebra. 
For example, pointed Hopf algebras have the Chevalley property.
In fact, there exist $8p$-dimensional 
Hopf algebras with the Chevalley property which are not pointed for each odd prime $p$,
see Proposition \ref{prop: lifting of kDic type (4,6)}. In this paper, 
in addition to some new general results,  we will provide a complete classification of 
nonsemisimple Hopf algebras of dimension 
$8p$ with $p$ an odd prime, which have the Chevalley property. Even though this result brings us one step 
closer to the complete classification of 24-dimensional nonsemisimple Hopf algebras, it remains unknown whether 
there exists any Hopf algebra $H$ of dimension 24 such that neither $H$ nor $H^*$ has the Chevalley property. 
Nevertheless,  some new observations of 24-dimensional Hopf algebras are presented in the last section of 
the paper which could be essential to the complete classification.

The paper is organized as follows. In Section \ref{sec:preliminaries} we introduce
several definitions and conventions that will be needed along the paper. We also provide some descriptions
on 
Hopf algebras with respect to their coalgebra structure and study pointed Hopf algebras over 
the group 
given by the semidirect product
$ \Ga_{4p} = C_p \rtimes_\chi C_4$. 
In Section \ref{sec:nonsemis8p} we collect some facts on nonsemisimple Hopf algebras 
of dimension $8p$
and prove our main result in Theorem \ref{thm:pointed8p}. The last section is devoted to the study 
of nonsemisimple Hopf algebras of dimension 24.
 
\section{Preliminaries}\label{sec:preliminaries}
\subsection{Definitions and Conventions}
We
work over an algebraically closed field $\Bbbk$ of characteristic
zero. Let $H$ be a finite-dimensional Hopf algebra over $\Bbbk$ with 
antipode $\Sc$. We use Sweedler's notation $\Delta(h)=h_1\ot
h_2$ for the comultiplication in $H$, but dropping the summation
symbol, see \cite{S}. The readers are referred to \cite{Mo} for details on
the basic definitions of Hopf algebras.

The \emph{coradical}
$H_0$ of $H$ is the sum of all simple
subcoalgebras of $H$. In particular, if $G(H)$ denotes the group
of \emph{group-like elements} of $H$, we have $\Bbbk G(H)\subseteq
H_0$. We write $H_{0} = \bigoplus_{d\geq 1} H_{0,d}$, where $H_{0,d}$ denotes the
direct sum of the simple subcoalgebras of $H$ of dimension $d^{2}$. By the Nichols-Zoeller
Theorem, its dimension is divisible by the order of the group of group-like elements (see \cite{andrunatale}). 

We say that a Hopf algebra $H$ is \emph{pointed} if $H_0=\Bbbk
G(H)$, and it is called \textit{basic} if all the simple $H$-modules are $1$-dimensional. 
Since $H$ is finite-dimensional, 
$H$ is basic if and only if $H^{*}$ is pointed.
 
We denote by $\{H_i\}_{i\geq 0}$ the \emph{coradical
filtration} of $H$. If $H_0$ is a Hopf
subalgebra of $H$, then 
$$\gr H = \bigoplus_{n\ge 0}\gr H(n)$$
is the associated
graded Hopf algebra, where $\gr H(n) = H_n/H_{n-1}$ for $n \ge 0$ and 
$H_{-1} =0$.  Let $\pi:\gr H\to H_0$ be the homogeneous
projection. Then 
$$R= (\gr H)^{\co \pi}:=\{r \in \gr H\mid (\id \o \pi) \Delta(r) = r \o 1\}$$
is called the \emph{diagram} of
$H$, which is a braided Hopf algebra in the category $\ydho$
of left Yetter-Drinfeld modules over $H_{0}$, and it is a
graded subalgebra of $\gr H$. The linear subspace  
$R(1)$ of homogeneous elements of $R$ of degree one, with
the braiding from $\ydho$, is called the \emph{infinitesimal
braiding} of $H$ and coincides with the subspace of primitive
elements 
$$\mathcal{P}=\{r \in R \mid \Delta_{R}(r) = r\ot 1 + 1\ot r\}\,,$$
where $\Delta_R$ denotes the comultiplication of $R$.
It turns out that the Hopf algebra $\gr H$ is the
Radford biproduct (or Magid bosonization) $\gr H \simeq R\# H_{0}$ (cf. \cite{Ra, Ma}). The
subalgebra of $R$ generated by $R(1)=V$ is isomorphic to the Nichols
algebra $\toba(V)$. If $H_{0}=\Bbbk G$ is a group algebra,
we denote the category of Yetter-Drinfeld modules simply by
$\ydg$.
Specifically, a left Yetter-Drinfeld module over $\Bbbk G$ is a finite-dimensional $G$-graded $\Bbbk$-linear 
space $M =
\oplus_{h\in G} M_h$ with a left $\Bbbk G$-module $(M,\cdot)$ structure such that
$$
\delta(g\cdot m) = ghg^{-1} \otimes g\cdot m, \qquad \forall\ m\in M_h, g, h\in G.
$$
Here $\delta$ is left $\kG$-comodule structure on $M$ given by the $G$-grading, namely with 
$\delta(m) =h \o m$ for $m \in M_h$.

Yetter-Drinfeld modules over $G$ are
completely reducible and the irreducible ones 
are parametrized by pairs $(\oc, \rho)$, where $\oc$ is a
conjugacy class of $G$ and $(\rho,V)$ is an irreducible representation of the
centralizer $C_G(\sigma)$ of a fixed element $\sigma\in \oc$: 
Then the irreducible Yetter-Drinfeld  module $M(\oc, \rho)$  is the induced  $\kG$-module 
$\kG \o_{\k C_G(\sigma)} V$.    Let $1=g_1, \dots , g_n$ be a complete set of left coset 
representatives of $C_G(\sigma)$ in $G$. Then
$\sigma_i = g_i \sigma g_i\inv$, $i=1,\dots, n$, are all the elements of  $\oc$ and 
$M(\oc, \rho) = \bigoplus_{i \le 1 \le n} g_i \o V$.  The $\kG$-comodule structure on 
$M(\oc, \rho)$ is given by $\delta(g_i v) = \s_i$ for all $i=1, \dots, n$, where $g_i v$ 
denotes the element $g_i \o v \in g_i \o V$. In particular, $\dim M(\oc, \rho) = |\oc|\dim V$.
We
denote the corresponding Nichols algebra by $\toba(\oc,\rho)$. For the definition of Nichols
algebras and basic facts, we refer the reader to \cite{A}.

Throughout this paper,  the cyclic group of order $\ell$ is denoted by $C_{\ell}$.
For $n\in \N$, $n\geq 3$, the dihedral group of order $2n$ is denoted by $\Dih_{n}$; 
we will use the  presentation
 $$
 \Dih_{n}=\langle a, y\mid a^2 = y^{n} =1, \,aya = y^{-1}\rangle.
 $$
 The dicyclic group of order $4n$, denoted by $\Dic_n$, has the presentation
  $$
 \Dic_n=\langle x, y\mid x^4 = y^{n} =1, \,xyx\inv = y^{-1}\rangle.
 $$

Let $N$ be a positive integer and let
$q \in \k$ be a primitive $ N$-th  root of unity. The \textit{Taft algebra} $ T_{q} $
is the $\k$-algebra generated by the elements $
g $ and $ x $ satisfying the relations $ x^{N} = 0 = 1-g^{N}$, $
gx=q xg $. It is a Hopf algebra with its coalgebra structure determined by 
$g$  being group-like and
$x$ a $(1,g)$-primitive, \textit{i.e.},  $ \com(g) = g\ot g $ and $
\com(x) = x\ot 1 + g\ot x $.  
It is self-dual and pointed of dimension
$ N^{2} $. If $N=2$ so that $q = -1$, then $T_{-1}$
is called the \textit{Sweedler algebra}
and will be denoted by $H_4$ throughout this article.

Recall that a tensor category $\mathcal{C}$  over $\Bbbk$
has the Chevalley property if the tensor product of any two simple
objects is semisimple. We shall say that a Hopf algebra $H$ has
the \emph{Chevalley property} if the category $\Comod (H)$ of
$H$-comodules does. In particular, $H$ has the Chevalley property if and only if
the coradical is a Hopf subalgebra. 

\begin{remark}
In some literature (for example \cite{NN}), $H$ is said to have the Chevalley property if $H$-mod does. 
In this case, $H$ has the Chevalley property if the Jacobson radical of $H$ is a Hopf ideal.
\end{remark}

\subsection{On simple coalgebras, group-likes and skew-primitives}
In this subsection we collect some structural results that will be needed later.

Recall that $H$ is a free left and a free right  $H^*$-module of rank 1  via the $H^*$-actions $\rightharpoonup$  
and  $\leftharpoonup$ defined by
$$
f \rightharpoonup h = f(h_2) h_1\quad\text{and}\quad h \leftharpoonup f  = f(h_1) h_2,
\qquad\text{ for all }f \in H^*\text{ and }h \in H. 
$$
Thus, a finite-dimensional $\k$-linear space $V$ is a right 
(resp. left) $H$-comodule 
if and only if $V$ is a left (resp. right) $H^*$-module.  If $\sigma$ is an algebra automorphism of $H^*$, 
then $V$ can be turned into another left $H^*$-module $_\s V$ with the $H^*$-action twisted by $\s$, namely
$$
f \cdot v := \s(f) v,\qquad \text{ for } f \in H^*\text{ and } v \in V.
$$

\begin{lem} \label{l:order 2}
  Let $C$ be a 4-dimensional simple coalgebra over $\k$ and $T:C \to C$ a 
  coalgebra automorphism with finite order. 
  If $\Tr (T) = 0$, then $\ord(T)=2$.
\end{lem}

\begin{proof}
  Let $n = \ord(T)$. By \cite[Theorem 1.4 (b)]{Ste97}, there is a primitive $n$-th root $\w \in \k$ and a basis 
  $\{e_{ij}\}_{i,j=1,2}$ of $C$ such that $T(e_{ij})=\w^{i-j} e_{ij}$. Thus, $\Tr(T) = (1+\w)(1+\w\inv)$ 
  which is zero, if and only if, $\w = -1$. 
  Therefore, $n=2$.
\end{proof}

\begin{lem}\label{l:inv1}
  Let $H$ be a finite-dimensional Hopf algebra that admits a decomposition of coalgebras
  $H=C\oplus C'$ with $C$ simple. 
  Then $C$ is invariant under $\Sc^{-2} \circ L(a)$,
  where $a$ is the distinguished group-like element of $H$ 
  and $L(a)$ denotes the left multiplication operator on $H$ by $a$. Moreover, 
  $\Tr_C(\Sc^{-2} \circ  L(a))=0$ provided $a$ is nontrivial. If,
  in addition, $\dim C= 4$, then $\Sc^{4}|_{C}=  L(a^{2})|_{C}$.
\end{lem}

\begin{proof} 
  Let $I$ be a simple right coideal of $H$ such that $\Delta(I) \subset I \o C$. 
  Then $C \cong \dim(I) I$ as right $H$-comodules and 
  hence left $H^*$-modules. In particular, $I$ is a simple projective $H^*$-module. 
  By \cite[Lemma 1.1]{Ng08}, we have that
  $I \cong \k_{\widehat{a\inv}} \o I^{\vee \vee}$, where $I^{\vee}$ denotes the left dual of $I$ 
  in the category of finite-dimensional
  $H^{*}$-modules, and $\hat x \in H^{**}$ denotes the image of $x\in H$ under the natural 
  isomorphism of $H\simeq H^{**}$. 
  In particular, $I \cong {_{\sigma} I}$ where $\sigma = (\Sc^*)^2 \circ R(\widehat {a \inv})$ 
  (cf. \cite[Lemma 1.2]{Ng08}). 
  Therefore, the annihilator ideal $\ann_{H^*} I$ is invariant under the automorphism $\sigma$. 
  Note that $\sigma = \tau^*$ 
  where $\tau=\Sc^2\circ L(a\inv)$ is a coalgebra automorphism. 
  Since $\ann_{H^*} I=C^\perp$, we find $\sigma(C^\perp)=C^\perp$ 
  and hence $C$ is invariant under $\tau$.

  Since $I \cong {_\sigma I}$ as $H^*$-modules, $C \cong {_\sigma C}$ as $H^*$-modules. To complete the proof, 
  by \cite[Lemma 1.2]{Ng08}, it suffices to show that $\tau\inv: C\to {_\sigma C}$ is a 
  $H^*$-module homomorphism. 
  For $f \in H^*$ and $c \in C$,
  $$
  \sigma(f) \rightharpoonup \tau\inv c \ = \sigma(f)(\tau\inv c_2)\tau\inv c_1 = f (\tau 
  (\tau\inv c_2))\tau\inv c_1=f (c_2)\tau\inv c_1 = \tau\inv (f\rightharpoonup c)\,.
  $$
  Therefore, $\tau\inv \in \Hom_{H^*}(C, {_\sigma C})$. 
  The last assertion follows then by Lemma \ref{l:order 2}.
\end{proof}

\begin{lem} \label{l:dis_gp}
  Let $H$ be a finite-dimensional Hopf algebra and $A$ a 
  Hopf subalgebra. Suppose $H$ admits a subcoalgebra decomposition  
  $H=A \oplus C$ for some subcoalgebra $C$ of $H$.
  Then $H$ and $A$ have a common distinguished group-like element.
\end{lem}

\begin{proof}
  Let $\l$ be a nonzero right integral of $H^*$. Then, for all $h \in H$, $h \leftharpoonup \l = \l(h) 1$ and 
  $\l \rightharpoonup h = \l(h) a$ where $a$ is the distinguished group-like element of $H$.  
  Note that for $c \in C$, we have $\l(c)=0$. Otherwise, 
  $\l(c) 1= \l(c_1) c_2 \in C$ and hence $1 \in C$, 
  which contradicts that $C \cap A=0$. Therefore, $\l|_A \ne 0$ and it is a right integral of $A$. 
  Let $h \in A$ such that $\l(h) \ne 0$. Then
  $$
  a = \l(h)\inv \l(h_2) h_1 \in A\,.
  $$ 
  Thus,  $a$ is the distinguished group-like element of $A$.
\end{proof}

The following proposition  is a refined version of \cite[Proposition 2.17]{BG} and \cite[Proposition 1.3]{HilNg09}.
\begin{prop}\label{prop:BGHilNg09}
  Let $\pi : H \to A$ be a Hopf algebra epimorphism and
assume $\dim H = 2 \dim A$. Then $H^{co \pi} = \k\{1, x\}$ with $x$ a $(1, g)$-primitive element, 
$g \in G(H)$ and $2 \mid \ord g $. Moreover, if $x$ is a trivial skew-primitive, \textit{i.e.}, a 
nonzero scalar multiple of $1-g$,
then $\ord g =2$ and $H$ fits into an exact sequence of Hopf algebras 
$\k C_2 \hookrightarrow H \twoheadrightarrow A$.
If $x$ is nontrivial, then $x^2=0$ and  $H$ contains a Hopf subalgebra of dimension $2 \ord (g)$; 
in particular, $4\mid \dim H$.\qed
\end{prop}

\subsection{On pointed Hopf algebras over $\Ga_{4p}$}\label{subsec:gamma4p}
Our main results will involve Hopf algebras whose coradical is a semisimple Hopf algebra of dimension $4p$ 
for $p$ an odd prime.
To this end, we need some results on the simple representations of the group $\Gamma_{4p}$.
Recall that for $p$ a prime congruent to $1 \mod 4$, $\Gamma_{4p}$ is defined to be
the group given by the semidirect product
$ \Ga_{4p} = C_p \rtimes_\chi C_4$ where
$C_p = \langle y \rangle$, $C_4 = \langle x\rangle$ are the cyclic groups of order $p$ and $4$, respectively, and  
$\chi(x)(y) = y^\ell$, where $\ell = \frac{p-1}{2}$. Moreover,  $\Ga_{4p}$ has the presentation
$$ \Ga_{4p} = \langle x,y:\ x^{4} = 1 = y^{p},\ xy = y^{\ell}x\rangle, $$
and it has $4 + \frac{p-1}{4}$ conjugacy classes and hence $\Ga_{4p}$ has $4 + \frac{p-1}{4}$ 
simple representations.
Among them, $4$ are one-dimensional, say $\alpha_{0},\ldots, \alpha_{3}$, 
which are given the quotient $\Ga_{4p}/C_p$. The remaining $\frac{p-1}{4}$  simple representations 
are 4-dimensional, say
$(\beta_{1},V_{1}), \ldots, (\beta_{\frac{p-1}{4}}, V_{\frac{p-1}{4}})$.

Denote by $\omega$
a primitive $4$th root of unity and by $q$ a primitive $p$th root of unity.

\begin{table}[ht]
\begin{center}
\begin{tabular}{|p{6cm}|p{3cm}|p{3.3cm}|p{2cm}|}
\hline {\bf Conjugacy class $\oc$} &    {\bf Centralizer} & {\bf Irrepn. $\rho$} &
$\dim M(\oc,\rho)$
\\ \hline  $e$   & $\Ga_{4p}$ & $\alpha_{j}$, $0\leq j \leq 3$ & $1$\\
\cline{3-4} & & $\beta_{j}$, $1\leq j \leq \frac{p-1}{4} $ & $4$
\\ \hline  $\Oc_{x^{m}} = \{x^{m}, yx^{m},\ldots, y^{p-1}x^{m} \}$, \newline
$\mid \Oc_{x^{m}} \mid = p$, $0<m<4$
& $\langle x \rangle \simeq C_{4}$ &
$\chi_j$, $\chi_{j}(x) = \omega^{j}$
\newline $0\leq j \leq 3$& $p \geq 5$
\\ \hline  $\Oc_{y^{k}} = \{y^{k},  y^{k\ell }, y^{k\ell^{2}}, y^{k\ell^{3}}\}$,
\newline $ \mid \Oc_{y^{k}} \mid =4$, $1<k\leq \frac{p-1}{4}$   &
$\langle y \rangle \simeq C_{p} $
& $\psi_{s}$,  $\psi_{s}(y) = q^{s}$
\newline $0\leq s<p$&  4
\\ \hline
\end{tabular}
\end{center}
\caption{$\Ga_{4p}$, with $p \equiv 1\mod  4$.}\label{tabladnpar:4p}
\end{table}

\begin{lem}\label{lem:nichols-infty}
Assume either $\oc$ is trivial or $\oc = \oc_{y^{k}}$ for $1<k\leq \frac{p-1}{4}$. Then
$\dim \toba(\oc, \rho)$ is infinite for $\rho$ any simple representation of $\Ga_{4p}$
or simple representation of $C_{p}$, respectively.
\end{lem}

\pf Assume first that $\oc=\{e\}$ and let $(\rho,V)$
be a simple representation of $\Ga_{4p}$. Then $M(\oc,\rho) = V$ is the $\kG$-module 
$V$ with the trivial $G$-grading. This implies that $c(v\ot w) = w\ot v$ for all
$v, w\in V$; in particular, $c(v\ot v) = v\ot v$. Consequently, $\dim \toba(\oc, \rho)$ is infinite,
since it contains the polynomial algebra in $v$.

Assume now that
$\oc = \oc_{y^{k}} = \{y^{k},  y^{k\ell }, y^{k\ell^{2}}, y^{k\ell^{3}}\}
= \{y^{k},  xy^{k}x^{-1}, x^{2}y^{k}x^{-2}, x^{3}y^{k}x^{-3}\} $
for some $k\neq 0$. Since
$C_{\Ga_{4p}}(y^{k}) = \langle y \rangle \simeq C_{p}$, all simple representations
are one-dimensional, say $\psi_{0},\ldots, \psi_{p-1}$ with $\psi_{s}(y) = q^{s}$.
Then $M(\oc_{y^{k}}, \psi_{s}) = \bigoplus_{j=0}^{3} x^{j}\ot \Bbbk 1$. If we
denote $w_{j} = x^{j}\ot 1$, we have that
$$ x\cdot w_{j} = w_{j+1},\qquad y\cdot w_{j} = q^{s\ell^{4-j}} w_{j},
\qquad \text{ and }\qquad\delta(w_{j}) = y^{k\ell^{j}}\ot w_{j}
\qquad\text{ for all }j\in \Z_{4}.$$
In particular, the formula for the braiding equals
$$c(w_{r}\ot w_{t}) = y^{k\ell^{r}}\cdot w_{t}\ot w_{r} =
q^{ks\ell^{r+4-t}} w_{t}\ot w_{r}
\qquad \text{ for all }0\leq r,t <4;$$
that is, $M(\oc_{y^{k}}, \psi_{s}) $ is a braided vector space of diagonal type.
If $s=0$, then $c(w_{r}\ot w_{t}) =
w_{t}\ot w_{r} $, which implies that $\dim \toba (\oc_{y^{k}}, \psi_{0}) = \infty$.
Assume $s\neq 0$ and
denote $\xi = q^{ks}$. 
Then by \cite[Theorem 17]{H}, we have that $\dim \toba (\oc_{y^{k}}, \psi_{s}) = \infty$
for all $0< s < p$.
\epf

\begin{obs} For
$0<m<4$, consider the conjugacy class $\oc = \oc_{x^{m}}=\{x^{m}, yx^{m}y^{-1},\ldots, y^{p-1}x^{m}y^{1-p} \}$, 
and let $\chi_k$ be the one-dimensional representation of $\Ga_{4p}$ given by $\chi_{k}(x) = \omega^{k}$
for $0<k<4$. Then, $M(\oc_{x^{m}}, \chi_{k}) = \bigoplus_{j=0}^{p-1} y^{j}\otimes \Bbbk 1$;
in particular, $\dim M(\oc_{x^{m}}, \chi_{k}) = p$.
If we write $v_{j} = y^{j}\ot 1$ for $0\leq j <p$, we have that
$$ x\cdot v_{j} = \omega^{k}v_{j\ell +1},\qquad y\cdot v_{j} = v_{j+1},
\qquad \text{ and }\qquad \delta(v_{j}) = y^{j(1-\ell^{m})}x^{m}\ot v_{j}
\qquad\text{ for all }j\in \Z_{p}.$$
In particular, the braiding equals
$$c(v_{r}\ot v_{t}) = y^{r(1-\ell^{m})}x^{m}\cdot v_{t}\ot v_{r} =
\omega^{mk} v_{t\ell^{m}+ r(1-\ell^{m})}\ot v_{r}
\qquad \text{ for all }0\leq r,t <p.$$
Consequently, if $mk\equiv 0 \mod 4$, we have that
$c(v_{0}\ot v_{0}) = x^{m}\cdot v_{0}\ot v_{0} =
v_{0}\ot v_{0} $, which implies that $\dim \toba(\oc_{x^{m}}, \chi_{k})=\infty$.
If $mk\not\equiv 0 \mod 4$, pointed Hopf algebras with group of points $\Gamma_{4p}$ do occur,
see \cite{Gr-zoo}, \cite{AGr}. 
\end{obs}

\section{Nonsemisimple Hopf algebras $H$ of dimension $8p$}\label{sec:nonsemis8p}
Throughout this section, we assume that $H$ is a nonsemisimple Hopf algebra of dimension $8p$,
with $p$ an odd prime unless specified otherwise. 
For each $\b \in G(H^*)$,  the 1-dimensional $H$-module associated with $\b$ is denoted by $\k_\b$. 
In particular, $\k_\e$ is the trivial 1-dimensional $H$-module. We denote the Jacobson radical of $H$ by $J(H)$, 
and the projective cover of an 
$H$-module $V$ by $P(V)$.

\subsection{Nonsemisimple Hopf algebras $H$ of dimension $8p$ with $|G(H)|=2p$}
In this subsection we improve upon the results of \cite{BG2} for $H$ a nonsemisimple nonpointed 
Hopf algebra of dimension $8p$ and group likes of order $2p$, $p$ an odd prime.  
We begin by recalling the well-known structure of the pointed Hopf algebras of dimension $4p$.

\begin{obs}\label{rmks:pointed4p-8}
For $p$ an odd prime, nonsemisimple pointed Hopf algebras of dimension $4p$ were described in 
\cite[A.1]{andrunatale} and for $p=2$ the pointed Hopf algebras of dimension $8$ are well-known 
and were described for example in \cite{Ste97}.  
In the first case, all such Hopf algebras $H$ have $G(H) \simeq C_{2p}$.  
The pointed Hopf algebras of dimension $8$ can have group-likes of order $2$ or $4$ and, if order $4$, 
can have group of group-likes isomorphic to the Klein $4$-group or to the cyclic group of order $4$.   
In either case ($p$ odd or $p=2$), there are $4$ nonisomorphic Hopf algebras with group-likes 
of order $2p$ as listed below.  
Here $g,h$ are group-like elements, $x$ is skew-primitive and $\lambda$ is a primitive $2p$th root of unity.  
\begin{align*}
\matha(-1,0):=&\Bbbk\langle g,x  \mid
g^{2p}-1=x^2 =gx+xg=  0\rangle,\quad
\com(x)=x\ot 1  + g\ot x.\\
\matha(-1,0)^\ast:=&\Bbbk\langle g,x  \mid
g^{2p}-1=x^2 =gx -\lambda xg=  0\rangle,
\quad \com(x)=x\ot 1+g^p\ot x .\\
\matha(-1,1) :=&\Bbbk\langle g,x  \mid
g^{2p}-1= x^2 - g^2 + 1 =gx+   xg=  0\rangle,
\quad\com(x)=x\ot 1+g \ot x .\\
H_{4} \otimes \Bbbk C_p:=&\Bbbk\langle g,h,x | h^2 -1 = g^p -1 = x^2 = gh-hg = hx + xh = gx -xg =0\rangle,
\quad\com(x)=x\ot 1+h \ot x .
\end{align*}

In the case $p=2$,
 we write $\matha_{8}(-1,1)$ to emphasize it is $8$-dimensional. Note that $\matha_{8}(-1,1) =A_{4}''$
 in \cite{Ste97} notation.

The dual of $\matha(-1,1)$ is nonpointed and as a coalgebra 
is isomorphic 
 to $ H_4 \oplus \mathcal{M}^\ast(2, \Bbbk)^{p-1}$.
On the other hand, the Hopf algebras
$\matha(-1,0)$ and $\matha(-1,1)$ do not have Hopf subalgebras isomorphic to $H_4$ but
$\matha(-1,0)^\ast$ and $H_4 \otimes \Bbbk C_p$ do.
In all four cases, $\Sc^4 = \id$.  The distinguished group-like elements of $\matha(-1,0)$, $\matha(-1,0)^*$  
and $\matha(-1,1)$  are respectively $g$, $g$ and $g^p$ in the above convention. 

 \end{obs}

The following is known from \cite{BG2}.

\begin{prop} \cite[Propositions 4.5, 4.6, 4.7]{BG2}\label{prop:bg454647}
Let $H$ be a nonsemisimple nonpointed Hopf algebra of dimension $8p$ and group-likes of order $2p$, 
$p$ an odd prime. Then:  
\begin{enumerate}
 \item[$(a)$] $H$ contains a pointed Hopf subalgebra $A$ of dimension $4p$.  As a coalgebra $H \simeq 
A \oplus \M^*(2,\Bbbk)^p$ and $H_0 \simeq \Bbbk C_{2p} \oplus \M^*(2,\Bbbk)^p$ and so has dimension $6p$.
 \item[$(b)$] $H^*$ does not have a Hopf subalgebra isomorphic to $H_4$.
 \item[$(c)$] $|G(H^*)|$ is $2$ or $4$.
 \item[$(d)$]  Suppose either 
 \begin{enumerate}
  \item[$(i)$] $A^{*}$ is nonpointed (so that $A$ is isomorphic to $\matha(-1,1)$), or
  \item[$(ii)$] $A^*$ is pointed and $H^*$ has a nontrivial 
 skew-primitive,
 \end{enumerate}
 then $H^*$ has a Hopf subalgebra isomorphic to $\matha_{8}(-1,1)$.    
\end{enumerate}
     \end{prop}

Now we strengthen the results above. 
\begin{lem}\label{l:coalg_decomp}
Let $H$ be an $8p$-dimensional nonsemisimple Hopf algebra with  $|G(H)| = 2p$ 
such that $H$ and $H^*$ are nonpointed.    
Then
$H$ contains a pointed normal Hopf subalgebra $A$, isomorphic to either $\matha(-1,1)$ or $\matha(-1,0)$, 
and fits into  
the exact sequence $1 \to A \to H \to \Bbbk C_2 \to 1$,  which has the dual exact sequence
$1 \to \Bbbk C_2 \to H^* \xrightarrow{\pi} A^* \to 1$. 
Furthermore, $G(H^{*}) \simeq C_{4}$ and $H^*$ has a Hopf subalgebra $Y$ isomorphic to $\matha_{8}(-1,1)$
with image in $A^*$ isomorphic to $H_4$.  
\end{lem}

\begin{proof}
Let $A$ be the Hopf subalgebra of $H$ given by Proposition \ref{prop:bg454647} $(a)$, and
  write $\pi: H^* \to A^*$ for the induced surjection. 
  By Proposition \ref{prop:BGHilNg09}, $R={H^*}^{\co \pi}$ is either a 
  group algebra of dimension 2 or it is 
  generated by a nontrivial $(1,g)$-primitive $y$ with $y^2=0$ and $2 \mid \ord(g)$. 
  Assume the latter case.
As $R$ is left-normal, it follows that $h_1y\Sc(h_2)\in \Bbbk y$ for all 
$h\in H^*$. Thus, $hy=h_1y\Sc(h_2)h_3=cyk$ with $c\in \Bbbk$ and $k\in H^{+}$, which implies that
$H^*y\subset yH^*$. Therefore, $yH^* = R^+ H^*$ 
  is a Hopf ideal of $H^*$ and $H^*/R^+H^* \cong A^*$. In particular, $\dim R^+ H^* =4p$. 
 As $y^2=0$, it follows that $R^+H^* \subset J(H^*)$. 
  But since $\dim H_0 = 6p$, we have that 
  $\dim J(H^*) = 2p$, a contradiction. 
  Therefore, $R$ is a 2-dimensional normal Hopf subalgebra of $H^*$, and hence $A$ is a 
  normal Hopf subalgebra of $H$.
  
Let $e_0, \dots, e_{2p-1}$ be the primitive idempotents of $K=\k G(H)$. 
  Since $H$ is a free $K$-module, $\dim He_i = 4$. 
  Therefore, $\dim P \le 4$ for all principal projective $H$-modules $P$. 
  Suppose $H^*$ has only trivial skew-primitive elements. 
  Then $P(\k)$ must have a simple constituent $V$ with $\dim V > 1$. 
  This implies $\dim P(V^*) > 4$, 
  giving a contradiction and so showing that $H^{*}$ has a nontrivial skew-primitive. 
  Thus by Proposition \ref{prop:bg454647} $(d)$,
  $H^{*}$ contains a Hopf subalgebra $Y$ isomorphic to $\matha_{8}(-1,1)$.
  By the structure of $\matha_8(-1,1)$ and by Proposition \ref{prop:bg454647} $(c)$, 
  $G(H^*) \simeq \Bbbk C_4$.
  
Write $1 \to \Bbbk \langle \beta \rangle \to H^* \to A^* \to 1$ with $\beta$ a group-like
  element of order $2$. Suppose that $Y$ is generated by the $(1,\alpha)$-primitive $y$, 
  where $\alpha$ has order $4$.  Then $\beta = \alpha^2$ and $\pi(\beta) = 1$. 
  Thus the image of $G(H^*)$ under $\pi$ is $\Bbbk C_2\simeq \Bbbk\langle\pi(\alpha)\rangle$.  
  Since $y \notin H^{*co\pi}$, then $\pi(y) \neq 0$ so that  $\pi(y)$ is a nontrivial $(1,\pi(\alpha))$-primitive.
  and $\pi(A_8(-1,1))$ is isomorphic to $H_4$.   
  Then $A^*$ must contain a Hopf subalgebra isomorphic to $H_4$, so that by Remark \ref{rmks:pointed4p-8}, 
  $A$ is isomorphic to either $\matha(-1,1)$ or $\matha(-1,0)$.
\end{proof}

We end this subsection with the following theorems.

\begin{thm}
Let $H$ be an $8p$-dimensional nonsemisimple Hopf algebra with  $|G(H)| = 2p$ 
such that both $H$ and $H^*$ are nonpointed. 
Then $\ord(\Sc^2)=2p$.
\end{thm}

\begin{proof}
  By Lemma \ref{l:coalg_decomp} and Remark \ref{rmks:pointed4p-8}, $H$ contains a pointed
  $4p$-dimensional Hopf subalgebra $A$ generated by a group-like element $g$ of order $2p$ 
  and a $(1,g)$-primitive 
  element $x$, and 
  $H = A \oplus H_{0,2}$ as coalgebras. Note that, $\Sc^4|_A = \id_A$ and by  
  Lemma \ref{l:dis_gp}, $g$ is the distinguished group-like element of $H$. 
  Let $D$ be a simple subcoalgebra of dimension 4.  
  By Lemmas \ref{l:order 2} and \ref{l:inv1}, there exists a matrix basis 
  $\{e_{11}, e_{22}, e_{21}, e_{12}\}$ of $D$ 
  such that 
  $\Sc^2(e_{ij}) = (-1)^{i-j} g e_{ij}$. Since $G(H)$ acts on the 4-dimensional simple subcoalgebras of $H$, 
  $g^p D = D$ and so $\Sc^{2p}(D) = D$. 
  If $gD=D$, then $D$ is a left $(\k\langle g \rangle, H)$-Hopf module and hence $D$ must be free over 
  $\k\langle g\rangle$. This is not possible since $\dim D=4$. Therefore,  $g D \ne D$. 
  In particular, $\Sc^2(D) \cap D = 0$ and $p\mid \ord(\Sc^2) \mid 4p$. 
  
  Suppose $\Sc^{2p}|_D= \id$. Then the subalgebra $K$ of $H$ generated by 
  $\k\langle g\rangle\oplus D$ is a sub-bialgebra (and hence a Hopf subalgebra)  
  of dimension  $> 4p$. Therefore, $K=H$ and hence $\Sc^{2p}  =\id$. 
  This implies $\Sc^{2p}|_A=\id$. However, since $A$ is a pointed nonsemisimple 
  Hopf algebra of dimension $4p$,  $\ord(\Sc^2|_A)= 2$, a contradiction.  
  Therefore, $\ord(\Sc^{2}|_{H_{0,2}}) = 2p = \ord(\Sc^2)$.  
\end{proof}

\begin{thm}\label{thm:8pdimssnot8}
	Let $H$ be a Hopf algebra of dimension $2^np$ over $\k$,  where $n\geq 0$ and $p$ is an odd prime.  
	If $H$ has a semisimple Hopf subalgebra of dimension $2^n$, then $H$ is semisimple.
\end{thm}

\begin{proof}
	Suppose $H$ has a semisimple Hopf subalgebra $K$ of dimension
	$2^n$, and let $\Sc$ be the antipode of $H$. 
	We consider three cases to show that $H$ must be semisimple.

	First, suppose that $H$ contains a group-like element $a$ of order $p$.    
	Define $\overline{K}$ to be the $\k$-subalgebra of 
	$H$ generated by $C=K\oplus\k\left\langle a \right\rangle$.  
	Then, $\overline{K}$ is a Hopf subalgebra and $\dim{\overline{K}}>2^n$, and so $\overline{K}=H$.  
	Further, since $C$ is the direct sum of two semisimple Hopf algebras as a vector space, 
	$\Sc^2|_C=\id_C$ and so $S^2=\id_H$. Hence $H$ must be semisimple.

	 Next suppose that $H^*$ contains a group-like element $\a$ of order $p$. 
	 Let $\pi: H^* \to K^*$  be the Hopf algebra surjection dual to the inclusion of $K$ in $H$. 
	 Since $\ord(\pi(\a)) \mid \dim K$, $\pi(\a)=1$ and so 
	 $\a\in(H^*)^{\text{co} \pi}$. Thus $\k\left\langle\a\right\rangle\subset(H^*)^{\text{co} \pi}$.  
	 However, since $\dim (H^*)^{\text{co} \pi}= p$, we actually have 
	 $\k\left\langle\alpha\right\rangle=(H^*)^{\text{co} \pi}$ and thus the exact sequence
		\begin{equation*}
			1 \rightarrow \Bbbk C_p\rightarrow H^*\rightarrow K^*\rightarrow 1
		\end{equation*}
of Hopf algebras. Since $\Bbbk C_p$ and $K^*$ are semisimple, so is $H^*$ and thus so is $H$.

	Finally suppose that neither $H$ nor $H^*$ has any group-like element of order $p$. By 
	Radford's formula, we have  $(\Sc^4)^{2n}=\id_H$.  In particular, $\ord(\Sc^2)$ is a 2-power.  
	Let $e$ be the normalized integral of $K$.  Since $H$ is a free right $K$-module,  $\dim He=p$.
	
Suppose $H$ is not semisimple.  It follows from \cite[Lemma 3]{RadSch00} that $\Tr(\Sc^2\circ R(e))=0$ 
and hence $\Tr(\Sc^2|_{He})=0$.  However, we know $\dim{He}=p$ and the order of $\Sc^2$ is a 2-power, so 
by \cite[Lemma 1.8]{HilNg09}, $\Tr(\Sc^2|_{He})\ne 0$, a contradiction!  Therefore, $H$ is semisimple.
\end{proof}

\subsection{Nonsemisimple Hopf algebras of dimension $8p$ with the Chevalley Property}
Throughout this subsection, we assume that $H$ is a Hopf algebra of dimension $8p$, $p$ an odd prime.
If $H$ has the Chevalley property,  by the Nichols-Zoeller Theorem, the dimension $n$
of $H_0$ must be a divisor of $8p$, \textit{i.e.},  $n \in \{2,4,8, p,2p,4p,8p   \}$.
If $H$ is not pointed, since
the only semisimple Hopf algebras of dimension $2,4,p$ are group algebras, it follows that $n \in \{8, 2p,4p\}$.

Assume $L$ is a semisimple Hopf algebra of dimension $8$. By \cite{Mk95}, $L$ is isomorphic to one of the 
following Hopf algebras: 
\begin{itemize}
\item $\Bbbk C_{8}$, $\Bbbk (C_{2} \times C_{4})$,
$\Bbbk (C_{2} \times C_{2} \times C_{2})$ (abelian case),
\item $\Bbbk Q_{8}$, $\Bbbk\Dih_{4}$, where $Q_8$ is the quaternion group (nonabelian case),
\item $\Bbbk^{Q_{8}}$,  with  $\Bbbk^{Q_{8}}  = \Bbbk C_{4} \oplus
\coM$ as coalgebras (dual of non-abelian case),
\item $\Bbbk^{\Dih_{4}}$,  with  $\Bbbk^{\Dih_{4}}
= \Bbbk C_{4} \oplus \coM$ as coalgebras (dual of non-abelian case), or
\item $B_{8} =  \Bbbk(C_{2} \times C_{2}) \oplus \coM$ as
coalgebras, see the definition in Remarks \ref{rmk:semisimple4p} and \ref{rmk:B8}.
\end{itemize}

\par Also recall that every Hopf algebra of dimension $2p$ is a group
algebra or isomorphic to $\Bbbk^{\Dih_{p}}$, the
dual of the group algebra on the dihedral group of order $2p$.

\begin{prop}\label{prop:chev-prop-cor-4p}
Let $H$ be a nonsemisimple nonpointed  Hopf algebra of dimension $8p$.  
If $H$ has the Chevalley property, then $\dim H_0 = 4p$.
\end{prop}

\begin{proof}
 Since $H$ is nonsemisimple and nonpointed, we have that $\dim H_{0}$
cannot be $1,2,4,p$ or $8p$. Further, by Theorem \ref{thm:8pdimssnot8}, we also 
have that $\dim H_{0} \neq 8$. 
We show that $2p$ is also
impossible.

Assume $\dim H_0 = 2p$. Since $H$ is not pointed, we must 
have that $H_0 \cong \Bbbk^{\Dih_{p}}$.
Then $\gr H \cong R \# \Bbbk^{\Dih_{p}}$ with $R$ a connected graded Hopf algebra in 
$_{\Bbbk^{\Dih_{p}}}^{\Bbbk^{\Dih_{p}}}\yd$.
Let $V=R(1)$ be the vector space of homogeneous
elements of degree one and consider the subalgebra $\toba(V)$ of $R$ generated by $V$; it is a Nichols
algebra of dimension less than or equal to $4$. Since $_{\Bbbk^{\Dih_{p}}}^{\Bbbk^{\Dih_{p}}}\yd
\simeq\ _{\Dih_{p}}^{\Dih_{p}}\yd$ as braided tensor categories, $\toba(V)$ corresponds
to a Nichols algebra $\toba(W)$ of dimension less or equal than $4$ in $_{\Dih_{p}}^{\Dih_{p}}\yd$. 
For dihedral groups $\Dih_n$ with $n$ odd, 
by \cite[Theorem 4.8]{AHS} the only possible finite-dimensional Nichols algebra 
over a simple module in $_{\Dih_{n}}^{\Dih_{n}}\yd$ comes from the conjugacy class
$\{a, ay,ay^2,..., ay^{n-1}\}$ in $\Dih_n$
and the braiding induced from the category. 
So, if $W\in\ _{\Dih_{p}}^{\Dih_{p}}\yd$ and $\toba(W)$ is finite-dimensional, then $W$ 
must contain a simple module isomorphic to 
this simple module of dimension $p$. As $\toba(W)$ is graded and the
first two homogeneous components satisfy that 
$\toba(W)^{0}=\Bbbk$, $\toba(W)^{1}=W$, we have that $\dim \toba(W) > p$. 
But moreover, if the top degree of $\toba(W)$ is $m$, then by the Poincar\'e duality
$\dim \toba(W)^{m-i}= \dim \toba(W)^{i}$. So, at least, $\dim \toba(W) \geq 2 + p $, which is impossible.
\end{proof}

In the remainder of this subsection we first recall the construction of a family of Hopf 
algebras of dimension $8p$ due to \cite{CDMM}, and prove that 
any nonsemisimple, nonpointed and nonbasic Hopf 
algebra of dimension $8p$ with the Chevalley property is isomorphic to a member of this family.

\begin{obs}\label{rmk:semisimple4p}
 For $p \equiv 3 \mod 4$ the semisimple Hopf algebras of dimension $4p$ are:
 \begin{itemize}
 \item $\Bbbk \Gamma$ with $\Gamma$
 an abelian group of order $4p$;
 \item $\Bbbk \Dih_{2p}$ and $\Bbbk^{\Dih_{2p}}$, where
 $\Bbbk^{\Dih_{2p}} \cong \Bbbk [C_2 \times C_2] \oplus \coM^{p-1}$ as coalgebras;
 \item $\Bbbk \Dic_p$ and $\Bbbk^{\Dic_p}$, where 
 $\Bbbk^{\Dic_p} \cong \Bbbk C_4 \oplus \coM^{p-1}$ as coalgebras;
 \item $ A_{4p}$,  with  $A_{4p}  = \Bbbk C_{4}  \oplus \coM^{p-1}$
as coalgebras and $A_{4p} \simeq A_{4p}^{*}$ (non-trivial self-dual case);
\item $B_{4p}$,  with $B_{4p}  = \Bbbk (C_{2} \times C_{2})  \oplus
\coM^{p-1}$ as coalgebras and $B_{4p} \simeq  B_{4p}^{*}$.
 \end{itemize}

Here, $A_{4p}$ denotes the $\k$-algebra generated by the elements $a$, $s_{+}$ and $s_{-}$
satisfying the relations
\begin{equation}\label{eq:relA4p}
 a^{2}=1,\qquad as_{\pm} = s_{\pm} a, \qquad s^{2}_{\pm} =1, \qquad (s_{+}s_{-})^{p}=1.  
\end{equation}
Let $e_0 = (1+a)/2$ and $e_1 = (1-a)/2$. Then $A_{4p}$ is a Hopf 
algebra with 
\begin{align*}
\com(a)&=a\ot a,& \eps(a)&=1,&\cS(a)&=a\\
\com(s_{\pm})&=s_{\pm}\ot e_{0}s_{\pm} + s_{\mp}\ot e_{1}s_{\pm},& \eps(s_{\pm})&=1,&
\cS(s_{\pm})&=e_{0}s_{\pm} + e_{1}s_{\mp}.
\end{align*}
On the other hand, $B_{4p}$  is the Hopf algebra defined in the same way, but replacing the
last relation in \eqref{eq:relA4p} by $(s_{+}s_{-})^{p}=a$. For more details, see \cite[Definition 3.3]{Mk00}.  
These Hopf algebras were first constructed
by Gelaki, see \cite{Gel1}.

If $p \equiv 1 \mod 4$ then there is an additional nonabelian group of
order $4p$ given by the semidirect product $\Ga_{4p} = C_p \rtimes_\chi C_4$, see Subsection
\ref{subsec:gamma4p}.
 If $C_p = \langle y \rangle$ and $C_4 = \langle x\rangle$ then $\chi(x)(y) = y^{(p-1)/2}$.
\end{obs}

\begin{obs} \label{rmk:B8}
If $p=2$, the definition of 
$B_{4p}$ above still makes sense and $B_{8}$ 
is the semisimple Hopf algebra of dimension 8 due to Kac and Paljutkin.
\end{obs}

Let $N > 1$ be an integer and $q \in \k$ a primitive $N$th root of 1. Recall that for $n\in \N$ 
$$(n)_{q} =\frac{q^{n}-1}{q-1}, \qquad (n)_{q}! = (n)_{q}(n-1)_{q}\cdots (1)_{q}.$$
Let $R_{q}$ denote
the graded algebra $\Bbbk[y]/(y^{N})$ endowed with a coalgebra structure such that
$$d_{0} = 1,\qquad
d_{1} = y,\qquad
d_{2} =\frac{y^{2}}{(2)_{q}!},\quad
\ldots,\quad
d_{N-1} =\frac{y^{N-1}}{(N-1)_{q}!},$$
form a divided power sequence in the sense that
$\com(d_{n} ) = \sum_{i=0}^{n} d_{i}\ot d_{n-i}$ for $0\leq n <N$. The homogeneous component $R_{q} (n)$ is $\Bbbk y^{n} (= \Bbbk d_{n} )$ and is $0$ if
$n\geq N$.

 \begin{definition}\cite[Definition 2.1]{CDMM}\label{def:YD-datum}
A Yetter-Drinfeld (or YD) datum for $R_{q}$, is a triple
$(L, g, \chi)$ which consists of a Hopf algebra $L$ with bijective antipode, a group-like element
$g$ in
$L$ and an algebra map $\chi: L \rightarrow \Bbbk$ such that $\chi(g) = q$ and
$$(\chi \rightharpoonup h)g = g(h\leftharpoonup \chi) \text{ for all }
h \in L.$$
\end{definition}

Using the YD-datum $(L, g, \chi)$ for $R_q$, one can define the $L$-action and $L$-coaction on 
$R_q$ by:
$$
h\cdot y^n = \chi^n(h) y^n \quad\text{and}\quad \delta(y^n)=g^n \o y^n, \quad \text{for } n=0,\dots, N-1,
$$
so that $R_q$ is a graded braided Hopf algebra in $^{L}_{L}\mathcal{YD}$. Conversely, 
it can be seen that any structure that makes $R_{q}$
into a graded braided Hopf algebra in $^{L}_{L}\mathcal{YD}$ arises uniquely from a YD datum for $R_{q}$,
see \cite{andrussch}.
Note that by the above definition $q \neq 1$ and so $g \neq 1$.  Thus to have a YD datum, it is
necessary to have $g \neq 1$ in the center of $G(L)$.

\begin{exas}\label{ex:dicyclic}
$(a)$
Let $\Bbbk^{\Dic_{p}}$ be the function algebra on the dicyclic group of order
$4p$. 
By \cite{CDMM}, this commutative Hopf algebra is generated by the elements $a$, $x$ satisfying:
\begin{align*}
a^{2} & = 1, & \eps(a)&=1,& \com(a)& =a\ot a,& \Sc(a) & = a,\\
x^{2p} & = 1, & \eps(x)&=1,& \com(x) & =  x\ot e_{0}x + ax^{-1}\ot e_{1}x,
 &
\Sc(x) & = e_{0}x^{-1} + e_{1}x.
\end{align*}
Here, $e_0 = (1+a)/2$,  $e_1 = (1-a)/2$ and
$g = (e_{0} + \sqrt{-1}e_{1})x^{p} $ is a group-like
element of order 4 that generates $G(\Bbbk^{\Dic_{p}})$.

A YD-Datum for $R_{-1} = \Bbbk[y]/(y^{2})$ is given by the triple $(\Bbbk^{\Dic_{p}}, g,\chi)$, where
$\chi: \Bbbk^{\Dic_{p}} \to \Bbbk $ is the
algebra map defined by $\chi(g) = -1$, $\chi(a)=1$ and $\chi(x) = -1$.

$(b)$ Consider the semisimple Hopf algebra $A_{4p}$ from Remark \ref{rmk:semisimple4p}.
There exist two YD-Datum for $R_{-1} = \Bbbk[y]/(y^{2})$, see \cite[Theorem 4.3]{CDMM}, which are
isomorphic to each other. They are given by the pairs $(s_{+}(p),\chi_{2})$ and $(s_{+}(p)a, \chi_{3})$, where
$s_{+}(p)=s_{+}s_{-}s_{+}s_{-}\cdots $ ($p$-times) $\in G(A_{4p})$ and $\chi_{2}$, $\chi_{3}$ 
are defined by $\chi_{2}(a)= -1$, $\chi_{2}(s_{\pm})=1$ and
$\chi_{3}(a)=-1=\chi_{3}(s_{\pm})$.
\end{exas}

\begin{lem}\label{lem:RYD-datum}
Let $H$ be a nonsemisimple Hopf algebra of dimension $2n$ with antipode $\Sc$. 
Assume that $H$ has a semisimple Hopf subalgebra $K$ of dimension $n$. Then,
\begin{enumerate}
 \item[$(a)$] $H$ has the Chevalley property with $K=H_0$, $\ord(\Sc^2)=2$ and $n$ must be even.
\item[$(b)$] Consider the graded Hopf algebra $\gr H= R\# K$. Then  $\dim R = 2$, 
$R \cong R_{-1}$ as $\k$-algebras, 
and $R$ admits a YD-datum structure $(K, g, \chi)$.
\end{enumerate}
\end{lem}

\pf
$(a)$ Denote by $\pi:H^*\rightarrow K^*$ the adjoint of the inclusion of $K$ in $H$ 
and set $R=(H^*)^{co \pi}$.  
Then by Proposition \ref{prop:BGHilNg09}, 
$R= \Bbbk\left\{1,x\right\}$, with $x$ a $(1,g)$-primitive element, $g\in G(H^{*})$ and $2\ |\ord(g)$.
If $x$ is trivial, then $H^{*}$ fits into an exact sequence of Hopf algebras 
$			0\rightarrow \k C_{2}\rightarrow H^*\rightarrow K^*\rightarrow 0
$,
which implies that $H^*$ is semisimple, a contradiction.  Thus, $x$ must be nontrivial. 
In such a case, we know that $x^2=0$ and $4 \mid \dim H^*$. Thus, $n$ must be even. 

Further, as $R$ is left-normal, it follows that $h_1r\Sc(h_2)\in R$ for all 
$h\in H^*$ and $r\in R$.  Moreover, one has that 
$\varepsilon(h_1x\Sc(h_2))=\varepsilon(x)\varepsilon(h)=0$. Since 
$R^{+}=\Bbbk x$, we conclude $h_1x\Sc(h_2)\in\k x$.	  
Thus, $hx=h_1x\Sc(h_2)h_3=cxh$ where $c$ is a scalar, implying $H^*x\subset xH^*$. But in such a case,
$(R^+H^*)^2 = (xH^*)^2=xH^*xH^*=x^2H^*=0$ so $R^+H^* \subseteq J(H^*)$. As
$K$ is semisimple, it holds that $J(H^*)\subset R^+H^*$, from which follows that   $J(H^*)= R^+H^*$.
Then, $K^*\cong H^*/R^+H^*$ and hence $K\cong (H^*/J(H^*))^*=H_0$.

Since $H$ is nonsemisimple, $0= \Tr(\Sc^2)$. 
Moreover, as $\Sc^2|_K=\id$, $\Sc^2$ induces a linear automorphism $\ol \Sc^2$ on $H/K$ and we have
$$
0=\Tr(\Sc^2) = \Tr(\Sc^2|_K)+ \Tr(\ol \Sc^2) = n+\Tr(\ol \Sc^2)\,. 
$$
Therefore, $\Tr(\ol \Sc^2)=-n$.
Note that $\dim(H/K)=n$ and $\ol \Sc^2$ is of finite order. 
This forces $\ol \Sc^2 = -\id$ and so the eigenvalues of $\Sc^2$ are $\pm 1$. 
Thus, we have $\Sc^4=\id$. Since $H$ is nonsemisimple, $\ord(\Sc^2)=2$. 

$(b)$ Since $H$ is not semisimple,
$\gr H$ is not semisimple.
Moreover, as
 $R=(\gr H)^{\co\pi}$ with $\pi: \gr H \twoheadrightarrow K$ the
 canonical projection, we have that
 $ \dim R=2$.   Then by Proposition \ref{prop:BGHilNg09}, $R$ has basis
$\{1,x\}$ where $x$ is a nontrivial $(1,g)$-primitive and $x^2 = 0$.  
 Then
we must have that $R \simeq R_{-1}$ and $x \in R(1)$. Since $R(1) \in \ydK$, there exists an algebra map  $\chi : K \to \k$ such that
$$ h\cdot x = \chi (h) x \text{ for all }h \in K. $$
Then by \cite[Lemma 8.1]{andrussch} and the following remark
we have that $(K, g, \chi)$ is a YD-Datum for $R_{-1}$ (see Definition \ref{def:YD-datum}).

\epf

\begin{definition} (cf. \cite{CDMM})
Define $\Hc_{8p}$ as the $\k$-algebra generated by the elements
$a,x,z$ satisfying the relations
\begin{align*}
a^{2} & = 1,&
x^{2p} & = 1,&
z^{2} & = {a -1}, \\
ax & = xa, &
az & = za, &
xz & = -zx.
\end{align*}
\end{definition}
\noindent It is a Hopf algebra with the coalgebra structure and antipode given by
\begin{align*}
\com(a)& =a\ot a,
& \com(x) & =  x\ot e_{0}x + ax^{-1}\ot e_{1}x,
& \com(z) & = g\ot z + z\ot 1\\
\Sc(a) & = a, &
\Sc(x) & = e_{0}x^{-1} + e_{1}x, &
\Sc(z) & = -g^{-1}z\\
\eps(a)& =1,&
 \eps(x) & =1,&
\eps(z) & =0,
\end{align*}
where $g = (e_{0} + \sqrt{-1}e_{1})x^{p} $. 
Clearly, $\Hc_{8p}$ contains $\Bbbk^{\Dic_{p}}$ as a Hopf
subalgebra and $g$ is a group-like element of order 4 that  generates $G(\Hc_{8p})$.
Moreover, by \cite[Theorem 4.2]{CDMM} we have that
$\Hc_{8p}$ is a lifting of a two-dimensional Nichols
algebra over $\Bbbk^{\Dic_{p}}$.

\begin{remark}
Note that the definition above is denoted $\Hc_{8p}(1)$ in \cite{CDMM}.  
There, the authors define the Hopf algebra $\Hc_{8p}(\alpha)$ as above except
that $z^2 = \alpha(a-1)$.  Since we are working over an algebraically closed field of characteristic 0, 
$\Hc_{8p}(\alpha)$ is isomorphic to $\Hc_{8p}(1)$ for $\a\ne 0$, and we denote $\Hc_{8p}(1)$ simply by 
$\Hc_{8p}$. Notice that $\Hc_{8p}(0) \simeq R_{-1}\# \Bbbk^{\Dic_{p}}$.
\end{remark}

\begin{prop}\label{prop: lifting of kDic type (4,6)} The Hopf algebra
$\Hc_{8p}$ is a nonsemisimple, nonpointed, nonbasic Hopf algebra
of dimension $8p$ with coradical $\Bbbk^{\Dic_{p}}$. Moreover,
$(\Hc_{8p}^{*})_{0} = \Bbbk C_{2p} \oplus \coM^{p}$.
\end{prop}

\pf By \cite[Theorem 4.2]{CDMM}, $\Hc_{8p}$ is a non-trivial lifting of a quantum
line $R_{-1}$ over the semisimple Hopf algebra $\Bbbk^{\Dic_{p}}$.
Hence it is nonsemisimple and nonpointed. Since the diagram $R_{-1}$
is given by the YD-Datum $(\Bbbk^{\Dic_{p}}, g,\chi)$, with
$\chi: \Bbbk^{\Dic_{p}} \to \Bbbk $
an algebra map defined by $\chi(g) = -1$, $\chi(a)=1$ and $\chi(x) = -1$, we have that
$\dim R_{-1} = 2$ and consequently, $\dim \Hc_{8p}= 8p$.

To finish the proof, we find all simple representations of $\Hc_{8p}$.
Since not all simple representations are one-dimensional, the
algebra is nonbasic.

Let $\xi$ be a primitive $2p$th root of unity. Since $\Bbbk^{\Dic_{p}}$ is a
commutative algebra, all simple representations are one-dimensional. They are
given by the family of $\Bbbk$-vector spaces
$(V_{i,j} = \Bbbk v_{i,j})_{i\in \Z_{2}, j\in \Z_{2p}}$
with the action determined by $a\cdot v_{i,j} = (-1)^{i}v_{i,j}$
and $x\cdot v_{i,j} = \xi^{j}v_{i,j}$.

The one-dimensional representations of $\Hc_{8p}$ are given
by the vector spaces $W_{j} = \Bbbk w_{j}$ with $j\in \Z_{2p}$. The action
is determined by $a\cdot w_{j} = w_{j}$, $x\cdot w_{j} = \xi^{j}w_{j}$
and $z\cdot w_{j} = 0$. Clearly, $W_{j}\simeq W_{\ell}$ if and only if
$\ell=j $ in $\Z_{2p}$.

The two-dimensional representations of $\Hc_{8p}$ are given
by the vector spaces $U_{i} = \Bbbk u_{1}^{i} + \Bbbk u_{2}^{i}$
with $i\in \Z_{2p}$. The action
is determined by
$$
a\cdot u_{j}^{i}  = - u_{j}^{i} \qquad  
x\cdot u_{j}^{i}  = (-1)^{j-1}\xi^{i}u_{j}^{i},\qquad
z\cdot u_{1}^{i}  = u_{2}^{i},\quad \text{ and }\quad
z\cdot u_{2}^{i} = -2 u_{1}^{i},
$$
for $j=1,2$. We show now that $U_{i}$ is simple for all $i\in \Z_{2p}$ and
$U_{i}\simeq U_{j}$ if and only if
$i= j $ or $i=j+p$ in $\Z_{2p}$.
Clearly, $U_{i} = V_{1,i} \oplus V_{1,i+p}$ as $\Bbbk^{\Dic_{p}}$-modules
with $V_{1,i} = \Bbbk u_{1}^{i}$ and $V_{1,i+p} = \Bbbk u_{2}^{i}$.

 Let
$i \in \Z_{2p}$ and assume that $U_{i}$ is not simple. Since it is two-dimensional,
it must contain a one-dimensional simple module, say spanned by
$v= \beta u_{1}^{i} + \gamma u_{2}^{i}$. Thus,
$a\cdot v = -v$
and $x\cdot v = \mu v$ for some $2p$th root of unity $\mu$.
Since $x \cdot v=\beta \xi^{i}u_{1}^{i} - \gamma \xi^{i}u_{2}^{i}$,
it follows that $\mu = \xi^{i}$ and $v = \beta u_{1}^{i}$ or
$\mu = -\xi^{i} = \xi^{i+p}$ and $v = \gamma u_{2}^{i}$.
But this is impossible, because
$z$ would not stabilize $v$.

Let $\varphi: U_{i} \to U_{j}$ be a morphism of $\Hc_{8p}$-modules.
Since both modules are simple, we have that $\varphi = 0$ or $\varphi$ is an isomorphism.
Assume that $\varphi \neq 0$. Then $\varphi(u_{1}^{i})$ must be an eigenvector
of $x$ in $U_{j}$ of eigenvalue $\xi^{i}$, since the eigenvalues of $x$ in
$U_{j}$ are $\pm\xi^{j}$, the claim follows. Consequently, we have at least
$p$ isomorphism
classes of two-dimensional simple modules.

Then $\dim (\Hc_{8p}^{*})_{0} \geq 6p$ with group-likes of order at least $2p$.
  By \cite[Proposition 3.2]{BG}, $\Hc_{8p}^{*}$ has a nontrivial
skew-primitive element and so  the dimension of the linear space spanned by the nontrivial
skew-primitive elements is at least $2p$. 
Thus, $(\Hc_{8p}^{*})_{0} = \Bbbk C_{2p} \oplus \coM^{p}$
as claimed.
\epf

\begin{thm}\label{thm:chevnonpointed8p}
Let $H$ be a nonsemisimple, nonpointed, nonbasic Hopf algebra
of dimension $8p$ with the Chevalley property. Then $H$ is isomorphic either to $R_{-1}\# A_{4p}$ or
$\Hc_{8p}$.
\end{thm}

\pf By
Proposition \ref{prop:chev-prop-cor-4p},
we know that $\dim H_{0} = 4p$. If we denote $A = \gr H$, then
we have that $A= R\#H_{0}$ and by Lemma \ref{lem:RYD-datum},
$R=R_{-1}$ admits a YD-datum given by
$(H_{0},g,\chi)$ for some $g\in G(H_{0})$ and $\chi: H_{0} \to \Bbbk$.
If $p \equiv 3 \mod 4$, we have four possibilities (up to isomorphism) for
$H_{0}$: $A_{4p}$, $B_{4p} $, $\Bbbk^{\Dih_{2p}}$ and $\Bbbk^{\Dic_{p}}$.
If $p \equiv 1 \mod 4$, there is an additional possibility given by $\Bbbk^{\Ga_{4p}}$.

First we show that the cases $B_{4p}$, $\Bbbk^{\Dih_{2p}},\Bbbk^{\Ga_{4p}}$ are
impossible.
By  \cite[Theorem 4.4]{CDMM}, there is no YD-datum for $R$ when
$H_{0} = B_{4p}$ and so the first case is impossible.
For $H_{0} = \Bbbk^{\Dih_{2p}}$, by \cite[Theorem 4.1]{CDMM}
there is one YD-datum but it admits no lifting. This implies
that $H = A = R_{-1}\#  \Bbbk^{\Dih_{2p}}$ from which follows
that $H^{*} \simeq R_{-1}^{*}\# \Bbbk \Dih_{2p}$ is pointed, a contradiction.
Finally, let $H_{0}= \Bbbk^{\Ga_{4p}}$ and assume that there exists a braided
Hopf algebra $R$ of dimension $2$ such that $A = R\# \Bbbk^{\Ga_{4p}}$. Then
$A^{*} \simeq R^{*}\# \Bbbk {\Ga_{4p}}$, where $R^{*}$ is a braided Hopf
algebra in $\ydgap$. Since this category is semisimple, $R^{*}$ is a sum of simple
modules. But since the simple modules in $\ydgap$ have dimensions $1, 4$ or $p$
by Table \ref{tabladnpar:4p}, $R^{*}$ must be the sum of two simple one-dimensional modules of
the form $M(\oc, \rho)$ with $\oc = \{1\}$.
This is impossible, because $R^{*}$ would contain the Nichols algebra
generated by these modules and the later are infinite-dimensional by
Lemma \ref{lem:nichols-infty}.

If $H_{0}= A_{4p}$, then by \cite[Theorem 4.3]{CDMM} $R$ admits a unique (up to isomorphism)
 YD-datum and it has no lifting. Thus $H = A = R_{-1}\# A_{4p}$;
in this case, $H$ is a self-dual Hopf algebra.
If $H_{0} = \Bbbk^{\Dic_{p}}$, then by \cite[Theorem 4.2]{CDMM},
$A =R_{-1} \# \Bbbk^{\Dic_{p}}$
has a nontrivial lifting isomorphic to $\Hc_{8p}$.
\epf

We end this section with the classification of Hopf algebras of dimension $8p$ with the Chevalley property.
The determination of the pointed Hopf algebras in this case follows essentially from
the results of Nichols, Gra\~na (\cite{nichols}, \cite{Gr,Gr2}), and several authors that contributed 
to the solution of the abelian case, among them Andruskiewitsch, Schneider,
Heckenberger and Angiono, see for example  \cite{AA}.

\begin{thm}\label{thm:pointed8p}
Let $H$ be a nonsemisimple Hopf algebra of dimension $8p$ with the Chevalley property.
\begin{enumerate}
 \item[$(a)$] Assume $H$ is pointed. 
 Then $H$ is isomorphic to one of the Hopf algebras constructed
 in \cite{Gr}.
 \item[$(b)$] Assume neither $H$ nor $H^{*}$ is pointed. Then  $H$ is isomorphic either to $R_{-1}\# A_{4p}$ or
$\Hc_{8p}$.
\end{enumerate}
\end{thm}

\pf $(a)$ Assume $H$ is pointed. Then $|G(H)| \in \{2,4,8,2p,4p\}$. The cases
$|G(H)| = 2$ and $8$ are impossible by \cite[Theorem 4.2.1]{nichols} and \cite[Corollary 4.1]{Gr},
respectively,
because $\dim H$ would be a power of $2$. The case $|G(H)|=4$ is also impossible 
for the same reason, although the result is due to several authors working on
pointed Hopf algebras over abelian groups. For details see for example \cite{AA}.
Finally, if $|G(H)| = 2p$ or $4p$, then 
$H$ is isomorphic to one of the Hopf algebras constructed
 in \cite{Gr}, since $\frac{\dim H}{|G(H)|} < 32$.   


$(b)$ Assume $H$ is neither pointed nor basic. Then $H$ is isomorphic either to $R_{-1}\# A_{4p}$ or
$\Hc_{8p}$ by Theorem \ref{thm:chevnonpointed8p}.
 \epf

 \section{Nonsemisimple Hopf algebras of dimension 24}
 We end the paper with some structural results on nonsemisimple Hopf algebras of dimension 24. 
 Throughout this section $H$ will be a nonsemisimple Hopf algebra of dimension 24.
 We begin with the following lemma.
 
\begin{lem} \label{l:1} If $G(H)$ is nontrivial,  then $H$ has a nontrivial skew-primitive element.
\end{lem}
\begin{proof} It suffices to prove the same statement for $H^*$. 
Note that $H^*$ having the Chevalley property is equivalent to the 
$\k$-linear full subcategory generated by the simple $H$-modules being a fusion subcategory of $H$-mod.

 Suppose $G(H^*)$ is nontrivial and $H^*$ has only trivial skew-primitive elements. 
 Then $H$ has a nontrivial 1-dimension $H$-module, 
 and $\Ext(\k, \k_\b)=\Ext(\k_\b, \k)=0$ for all $\b \in G(H^*)$. 
 Thus, $P(\k)/J^2 P(\k)$ has a simple $H$-submodule $V$ such that $\dim V \ge 2$.

  (i) We first observe that $\dim V =2$ otherwise $\dim V \ge 3$ and so
  $$
  \dim H \ge |G(H^*)| \dim P(\k) + \dim V \dim P(V) \ge 2 \cdot 5 + 3 \cdot 7 > 24.
  $$

  (ii) We claim that if $W$ is another simple $H$-module with $\dim W >1$, then $W$ is 
  projective with $\dim W =2$. Since
  $$
  \dim H \ge |G(H^*)| \dim P(\k) + \dim V \dim P(V) + \dim W \dim P(W)  \ge 2 \cdot 4 + 2 \cdot 5 + \dim W \dim P(W)\,.
  $$
  This forces $\dim W \dim P(W) \le 6$, and hence $\dim W = \dim P(W) = 2$. 
  In particular, $W$ is projective.

  (iii) Since $V$ is the unique simple $H$-module with $\dim V> 1$ that is not projective, 
  $V^* \cong V$ and $\k_\b \o V \cong V \cong V \o \k_\b$ 
  for all $\b \in G(H^*)$. Since $P(\k_\b) \cong \k_\b \o P(\k)$, $V$ is a simple $H$-submodule 
  of $P(\k_\b)/J^2 P(\k_\b)$. 
  This implies that $\sum_{\b \in G(H^*)} \k_\b$ is an $H$-submodule of $P(V)/J^2 P(V)$, and hence
  $$
  \dim P(V) \ge |G(H^*)|+2\dim V = |G(H^*)|+4 \ge 6\,.
  $$

  (iv) $|G(H^*)|=2$ for otherwise $\dim P(V) \geq 7$ and so
  $$
  \dim H \ge |G(H^*)|\dim P(\k) + 2 \dim P(V) \ge 3 \cdot 4 + 2 \cdot 7 > 24.
  $$

  (v) We now claim that $V$ is the unique simple $H$-module with $\dim V > 1$. Suppose not. By (ii), there is a 
  2-dimensional simple projective $H$-module $W$. The inequality
  $$
  \dim H \ge 2 \dim P(\k) + 2 \dim P(V) + (\dim W)^2 \ge 2 \cdot 4 + 2 \cdot 6 + 2 \cdot 2 =24
  $$
  implies that $W$ is the unique simple projective $H$-module,
  $$
  \dim P(\k) =4 \quad\text{and} \quad \dim P(V)=6.
  $$
  Therefore, $W \cong W^*$ and $\k_\a \o W \cong W \cong  W \o \k_\a$, 
  where $\a \in G(H^*)$ is the nontrivial group-like element. 
  Since $W$ is projective, so is $W \o W^*$. Therefore, $W \o W^*$ is a direct 
  sum of indecomposable projective $H$-modules.
  However, $W \o W^*$ maps surjectively onto $\k$ and $\k_\a$ which implies 
  $P(\k) \oplus P(\k_\a)$ is a summand of $W \o W^*$. 
  This is absurd as $\dim (W \o W^*) =\dim P(\k) = \dim P(\k_\a) =4$.

  (vi)  Let $m$ be the multiplicity of $V$ as a composition factor of $P(\k)$. 
  Then, by (iii), $m$ is also the multiplicity of $V$ 
  as a composition factor of $P(\k_\a)$. Since $\k$ is algebraically closed of 
  characteristic zero, $\dim P(V)$ is the multiplicity of 
  $V$ in the composition factors of $H$ (cf. \cite[Theorem 54.19]{CR}).
  By (v), we have
  $$
  H = P(\k) \oplus P(\k_\a) \oplus 2 P(V).
  $$
 As $V$ is composition factor of $P(\k)$, we have that $m > 0$.
 Since $\dim \Soc(P(V)) = \dim V =2$, $\Soc(P(V))=V$ 
  and hence $[P(V):V] \ge 2$. 
  Therefore, $\dim P(V) \ge 2 m + 4$. Thus, the inequality
  $$
  \dim H = 2 \dim P(\k)+ 2 \dim P(V) \ge  2(2+2m)+2(2m+4)= 4(2m+3)
  $$
  forces $m=1$. In particular, $\dim P(\k)=4$. Now we have
  $$
  \dim H = 2 \cdot 4 + 2 \dim P(V)\,,
  $$
  and so $\dim P(V) = 8$.

  (vii) Since $P(\k_\b) \o V$ is projective for any $\b \in G(H^*)$ and it maps 
  surjectively onto $V$, $P(V)$ is a summand of $P(\k_\b) \o V$. 
  This implies $P(\k_\b) \o V \cong P(V)$ by comparing dimensions. 
  Similarly, $V \o P(\k_\b) \cong P(V)$.

  (viii) We now prove that $H^*$ having the Chevalley property means that the 
  $\k$-linear abelian full subcategory generated by $\k, \k_\a$ 
  and $V$ is closed under the tensor product of $H$-mod. Since
  $$
  \dim \Hom_H(P(\k_\b), V \o V) = \dim \Hom_H(V^* \o P(\k_\b), V) =\dim \Hom_H(P(V), V) = 1,
  $$
  the multiplicity of the $\k_\b$ as a composition factor of $V \o V$ is 1. 
  The dual basis map $\db : \k \to V \o V^*$ and the 
  evaluation map $\ev: V^* \o V \to \k$ are nonzero $H$-module maps, so are 
  $\k_\b \o \db$ and  $\k_\b \o \ev$ for all $\b \in G(H^*)$. 
  Therefore, both $\k$ and $\k_\a$ are summands of $V\o V$, and hence 
  $V \o V  \cong V \oplus \k \oplus \k_\a$. By (iii), the $\k$-linear 
  abelian full subcategory $\DD$ generated by $\k, \k_\a$ and $V$ is a fusion subcategory 
  of $H$-mod.

  (ix) By the reconstruction theorem, $\DD$ is monoidally equivalent to $\ol H$-mod for 
  some semisimple quotient Hopf algebra  
  $\ol H$ of $H$. Since $\dim \ol H = 1+1+2^2 =6$, it follows from the classification of 
  6-dimensional Hopf algebras that $\ol H \cong \k \mathbb{S}_3$. 
  Thus, $\k^{\mathbb{S}_3} \cong (H^*)_0$ but this contradicts that the smallest Hopf algebras with 
  the Chevalley property and the coradical isomorphic 
  to  $\k^{\mathbb{S}_3}$ is of dimension 72 \cite{AV1}. This completes the proof of this lemma.
\end{proof}

\begin{cor}
  If $H$ and $H^*$ do not have the Chevalley property, then both $H$ and $H^*$ have a nontrivial 
  skew-primitive element.
\end{cor}
\begin{proof}
  By \cite[Lemma 4.25]{BG}, $|G(H)|$ and $|G(H^*)|$ are even. 
  The statement then follows immediately from Lemma \ref{l:1}.
\end{proof}

\begin{prop}
      Let $H$ be a nonbasic Hopf algebra of dimension $24$ without the Chevalley property. Then 
      $H$ is generated by its coradical.
     \end{prop}
     
\pf
Let $K$ be the Hopf subalgebra of $H$ generated by the coradical $H_{0}$. As $H$ is nonpointed, 
$\dim H_{0} >4$.  Since 6-dimensional Hopf algebras are semisimple,   $\dim K = 8, 12$ or $24$. 
Since $H$ does not have the Chevalley property, then
$K\neq H_{0}$ and $K$ is nonsemisimple and nonpointed. Consider the standard filtration 
associated to $K$ and write $\gr H = R\# K$, see \cite{AC}.

Assume $\dim K=8$. In such a case, $\dim R =3$. 
Let $V$ be an indecomposable Yetter-Drinfeld submodule of $R(1)$ and consider
the subalgebra $L$ of $R$ generated by $V$. Note that $L$ is not
necessarily a Nichols algebra since it might contain primitive elements in degree
bigger than one. In any case, there is a projection $L\twoheadrightarrow \toba(V)$ and
consequently $\dim L \geq \dim \toba(V)$. This implies by \cite[Theorem A]{GJG} that $V$
is simple and $\dim V =1$ or $2$. The latter case is impossible since $\dim \toba(V) >3$.
In the former case, by \cite[Proposition 3.2]{GJG} it follows that $\dim L =2$, 
which is impossible, since $\gr H$ would contain a Hopf subalgebra of dimension $16$.

Assume now that $\dim K = 12$. In such a case, $\dim R=2$. By the classification result of Natale \cite{N} 
(see also \cite{CN}), 
there is a unique Hopf algebra 
(up to isomorphism)
whose coradical is not a Hopf subalgebra. It is the dual of a 
pointed Hopf algebra; denote it by $ \mathcal{C}$.
As before, let $V$ be an indecomposable Yetter-Drinfeld submodule of $R(1)$ and consider
the subalgebra $L$ of $R$ generated by $V$.
Since $\dim R(0)=1$ and so $\dim R(1)=1$. 
Let $x$ be a nonzero element of $R(1)$. Then $x^2=0$ and so $R \cong R_{-1}$. In particular, 
$R^{*} \cong R$ as algebras and coalgebras.
Thus by \cite[Theorem B]{X}, $H$ is isomorphic to bosonization 
$\bigwedge \Bbbk_{\chi} \# \mathcal{C}$. 
This implies that 
$H^{*}\simeq \bigwedge \Bbbk_{\chi} \# \mathcal{C}^{*}$ and whence $H$ would be basic, a contradiction.
\epf


\end{document}